\documentclass[12pt]{article}
\usepackage{amsmath,amssymb,amsfonts,amsthm}
\newcounter{item}[section]
\newcounter{kirshr}
\newcounter{kirsha}
\newcounter{kirshb}
\newenvironment{enumarab}{\setcounter{kirshb}{1}
\begin{list}{(\arabic{kirshb})}{\usecounter{kirshb}} }{\end{list}}
\newenvironment{athm}[1]{\vskip3mm\par\noindent
{\bf #1 }. \slshape }
{\upshape\par\vskip10pt minus3pt}
\newtheorem{theorem}{Theorem}[section]

\newtheorem{lemma}[theorem]{Lemma}
\newtheorem{corollary}[theorem]{Corollary}
\newenvironment{demo}[1]{\noindent{\bf #1.}\upshape\mdseries}
{\nopagebreak{\hfill\rule{2mm}{2mm}\nopagebreak}\par\normalfont}
\theoremstyle{definition}
\newtheorem{remark}[theorem]{Remark}

\newtheorem{example}[theorem]{Example}
\newtheorem{definition}[theorem]{Definition}

\def\Q{\mathbb{Q}}
\def\C{{\mathfrak{C}}}
\def\Fm{{\mathfrak{Fm}}}

\def\Nr{{\mathfrak{Nr}}}

\def\Fm{{\mathfrak{Fm}}}
\def\A{{\mathfrak{A}}}
\def\B{{\mathfrak{B}}}
\def\C{{\mathfrak{C}}}

\def\M{{\mathfrak{M}}}
\def\N{{\mathfrak{N}}}
\def\Sn{{\mathfrak{Sn}}}
\def\CA{{\bf CA}}

\def\RCA{{\bf RCA}}

\def\(R)RA{{\bf (R)RA}}

\def\Q{\mathbb{Q}}

 \def\CA{{\sf CA}}
\def\B{{\sf B}}

\def\tp{{\sf tp}}
 
\def\Nr{{\mathfrak{Nr}}}

\def\Nr{{\mathfrak{Nr}}}

\def\A{{\mathfrak{A}}}
\def\B{{\mathfrak{B}}}
\def\C{{\mathfrak{C}}}

\def\A{{\mathfrak{A}}}
\def\B{{\mathfrak{B}}}
\def\C{{\mathfrak{C}}}

\def\L{{\mathfrak{L}}}

\def\L{{\mathfrak{L}}}
\def\CA{{\bf CA}}

\def\RCA{{\bf RCA}}

\title{An instance of Vaught's conjecture using algebraic logic}
\author{Mohammad Assem and Tarek Sayed Ahmed}
\begin{document}
\maketitle
\begin{abstract} We indicate a way of distinguishing between (what we call) Henkin ultrafilters of locally finite cylindric
and quasi-polyadic algebras, for which, two ultrafilters are said to be \emph{distinguishable.} We give a result
about the  number of distinguishable ultrafilters in a given locally finite countable algebra.
In model theoretic  terms, our main result says
that \emph{for any first order theory $T$ in a countable language, with or {\it without} equality,
if it has an uncountable set of countable models that are pairwise distinguishable, then actually it has such a set of size $2^{\aleph_0}$.}
\end{abstract}

\section{Introduction}

The subject of this paper is an algebraic version of the work in \cite{sep}, and also an extension thereof;
we consider the number of countable models for a countable omitting a given set of non-isolated types, and we consider
Vaught's conjecture for infinitary extensions of first order logic, also counting what we call weak models omitting non isolated types.
We also study omitting types for finite variable fragments of first order logic; we give a proof to a result mentioned (without proof) 
in \cite{Sayed}.

In 1961, Robert Vaught asked the following question:  Given a complete theory in a countable language, is it the case
that it  either has countably many or $2^{\aleph_0}$ non-isomorphic countable models?
By the number of non-isomorphic countable models is meant the number of their isomorphism-types; that is
the number of equivalence classes of countable models w.r.t. the isomorphism relation between structures.
We shall just say ``the number of countable models"  to mean the number of their isomorphism-types.

The positive answer to the question is more commonly know as Vaught's Conjecture.
(Vaught;s conjecture has the reputation of being the most important open problem in
model theory.)\footnote{ However, some logicians do not agree to this sweeping statement.
Quoting Shelah on this: {\it Poeple say that settling  Vaught's conjecture is the most important problem in Model theory, because
it makes us understand countable models of countable theories, which are the most important models.
We disagree with all three statements.}}

Morley proved that the number of countable models is either less than or equal to the first uncountable cardinal ($\leq \aleph_1$)
or else it has the power of the continuum. This is the best known (general) answer to Vaught's question.
(We will give an algebraic proof of Morley's theorem below)
Later other logicians confirmed Vaught's conjecture in some special cases of theories, for example:

\begin{enumerate}
\item (Shelah \cite{93})$\omega$-stable theories;
\item (Buechler \cite{19})superstable theories of finite $U$-rank;
\item (Mayer \cite{69})o-minimal theories;
\item (Miller \cite{100})theories of linear orders with unary predicates;
\item (Steel \cite{98})theories of trees.
\end{enumerate}

There are also attempts  concerning special kinds of models to count and aso relations other than
isomorphisms between models. Vaught's conjecture can be translated to counting the number of orbits corresponding to the action of
$S_{\infty}$ the symmetric group on $\omega$ on the Polish space of countable models.
One way to obtain a positive result is to consider only isomorphisms induced by a {\it subgroup} $G$ of $S_{\infty}$
Vaught's conjecture has been confirmed when $G$ is solvable; the best result in this type of investigations, is the case when $G$ is a cli group.

Our work here is inspired by Gabor Sagi, who approached Vaught's conjecture using the machinery of algebraic logic.

Here we consider, what we believe is an interesting equivalence
relation between models weaker than isomorphism
and show that it has either countably many classes, or else continuum many
(actually we prove something stronger). We also show that the same applies for the restriction of this relation to
models omitting a given family of types (possibly uncountable but $< 2^{\omega}$).
Our results, formulated for locally finite cylindric and quasi polyadic algebras,
hold for languages with or without equality and  also for theories that  are not necessarily complete.

\section*{Notation} Our system of notation is mostly standard, but the following list may be useful.
Throughout, both $\omega$ and $\mathbb{N}$ denote the set of natural numbers
and for every $n\in\omega$ we have $n=\{0,\ldots,n-1\}.$ Let $A$ and $B$ be sets.
Then ${}^AB$ denotes the set of functions whose domain is $A$ and whose range is a subset of $B.$
In addition, $|A|$ denotes the cardinality of $A$ and $\mathcal{P}(A)$ denotes the power set of $A$, that is, the set of all subsets of $A$.
If $f:A\longrightarrow B$ is a function and $X\subset A,$ then $f^*(X)=\{f(x):x\in X\}$ and $f|_X$ is the restriction of $f$ to $X$.
Moreover, $f^{-1}:\mathcal{P}(B)\longrightarrow\mathcal{P}(A)$ acts between the power sets.

{\bf Layout of the paper} In section one we give the basic preliminaries and concepts from cylindric and quasi-polyadic algebras.
We will only be concerned with the
locally finite
case (those are the algebras corresponding in an exact sense to first order logic with and without equality, respectively).
In section 2 we give a purely algebraic proof of Morley's theorem. In section 3,
we formulate and prove our main result. In section 4, we consider the number of models omitting (possibly uncountable many) types.
In the final section, we approach the case of proper extensions of first order logic.

\section{Counting ultrafilters}

Let $\A$ be any Boolean algebra. The set of ultrafilters of $\A$ is denoted by $\mathcal{U}(\A)$. The stone topology  makes $\mathcal{U}(\A)$ a compact Hausdorff space; we denote this space by $\A^*$. Recall that the Stone topology has as its basic open sets the sets $\{N_x:x\in A\}$, where
$$N_x=\{\mathcal F\in\mathcal{U}(\A):x\in\mathcal F\}.$$
It is easy to see that if $A$ is countable, then $\A^*$ is \emph{Polish}, (i.e., separable and completely metrizable).

Now, suppose $\A$ is a locally finite cylindric or quasi-polyadic $\omega$-dimensional algebra with a countable universe.
Note that if $T$ is a theory in a countable language with (without) equality, then $CA(T)$, (respectively $QPA(T)$), satisfies these requirements.
Let $$\mathcal{H}(\A)=\bigcap_{i<\omega,x\in A}(N_{-c_ix}\cup\bigcup_{j<\omega}N_{s^i_jx})$$ and, in the cylindric algebraic case, let

$$\mathcal{H}'(\A)=\mathcal{H}(\A)\cap\bigcap_{i\neq j\in\omega}N_{-d_{ij}}.$$
Note, for later use, that $\mathcal{H}(\A)$ and $\mathcal{H}'(\A)$ are $G_\delta$ subsets of $\A^*$, and are nonempty,  as a matter of face it is dense-- this latter fact can be seen,
for example, from Theorem \ref{th1} below -- and are therefore Polish spaces; (see \cite{Kechris}).
Assume $\mathcal F\in \mathcal{H}(\A).$ For any $x\in A$, define the function $\mathrm{rep}_{\mathcal F}$ to be $$\mathrm{rep}_{\mathcal F}(x)=\{\tau\in{}^\omega\omega:s^+_\tau x\in \mathcal F\}.$$
We have the following results due to G. S\'agi and D. Szir\'aki; (see \cite{Sagi}).
\begin{theorem}\label{th1}
If $\mathcal F\in \mathcal{H}'(\A)$, (respectively $\mathcal{H}(\A)$), then $\mathrm{rep}_{\mathcal F}$ is a homomorphism from $\A$ onto an element of $Lf_\omega\cap Cs_\omega^{reg}$, (respectively $LfQPA_\omega\cap Qs_\omega^{reg}$), with base $\omega.$ Conversely, if $h$ is a homomorphism from $\A$ onto an element of $Lf_\omega\cap Cs_\omega^{reg}$, (respectively $LfQPA_\omega\cap Qs_\omega^{reg}$), with base $\omega$, then there is a unique $\mathcal F\in \mathcal{H}'(\A)$, (respectively $\mathcal{H}(\A)$), such that $h=\mathrm{rep}_{\mathcal F}.$
\end{theorem}
\begin{theorem}
Let $T$ be a consistent first order theory in a countable language with (without) equality. Let $\mathcal{M}_0$ and $\mathcal{M}_1$ be two models of $T$ whose universe is $\omega$. Suppose $\mathcal{F}_0,\mathcal{F}_1\in\mathcal{H}'(CA(T))$, (respectively $\mathcal{H}(QPA(T))$), are such that $\mathrm{rep}_{\mathcal{F}_i}$ are homomorphisms from $CA(T)$, (respectively $QPA(T)$), onto $Cs(\mathcal{M}_i)$, (respectively $Qs(\mathcal{M}_i)$), for $i=0,1$. If 
$\rho:\omega\longrightarrow\omega$ is a bijection,
then the following are equivalent:
 \begin{enumerate}
 \item $\rho:\mathcal{M}_0\longrightarrow\mathcal{M}_1$ is an isomorphism.
 \item $\mathcal{F}_1=s^+_\rho \mathcal{F}_0=\{s^+_\rho x:x\in \mathcal{F}_0\}.$
\end{enumerate}
\end{theorem}
These last two theorems allow us to study models and count them via corresponding ultrafilters. This approach was initiated by S\'agi.
The main advantage of such an approach is that results proved
for locally finite cylindric algebras transfer {\it mutatis mutandis} to quasi-polyadic algebras (without diagonal elements).
So from the algebraic point of view we do the difficult task only once,
but from the model theoretic point of view we
obtain deep theorems for first order logic {\it without} equality, as well, which are more often than not, not obvious to prove without
the process of algebraisation. In fact, our main result here has an easy metalogical proof when we have equality (see the concluding remark),
but this proof does not work  in the absence of equality.
However, the algebraic proof does.
Vaught's conjecture has been confirmed when we restrict the action on certain subgroups of $G$.
But in this  case there might be isomorphic models that the group $G$ does not 'see' (the isomorphism witnessing this
can be outside $G$) so the equivalence relation is drastically different.

\begin{theorem} Let $G\subseteq S_{\infty}$ be a cli group. Then $|{\cal H}(\A)/E_G|\leq \omega$ or $|{\cal H}(\A)/E_G|=2^{\omega}$
\end{theorem}
\begin{demo}{Proof}  It is known that the number of orbits of $E_G$  satisfies the so-called Glimm Effros Dichotomy.
By known results in the literature on the topological version of
Vaught's conjecture, we have ${\cal H}(\A)/E_G$
is either at most countable or ${\cal H}(\A)/E_G$ contains continuum many non equivalent elements
(i.e non-isomorphic models).
\end{demo}
It is known that the number of orbits of $E=E_{S_{\infty}}$ {\it does not} satisfy the Glimm Effros Dichotomy.
We note that cli groups cover all natural extensions of abelian groups, like nilpotent and solvable groups.
Now we give a topological condition that implies Vaught's conjecture.
Let everything be as above with $G$ denoting a Polish subgroup of $S_{\infty}$.
Give ${\cal H}(\A)/E_G$ the qoutient topology and let $\pi: {\cal H}(\A)\to {\cal H}(\A)/E_G$ be the projection map.
$\pi$ of course depends on $G$, we somtimes denote it by $\pi_G$ to emphasize the dependence.


\begin{lemma} $\pi$ is open.
\end{lemma}
\begin{proof}
To show that $\pi$ is open it is enough to show for arbitrary $a\in \A$ that $\pi^{-1}(\pi(N_a))$ is open. For,
\begin{align*}
\pi^{-1}(\pi(N_a))&=\{F\in\A^*:(\exists F'\in N_a)(F,F')\in E\}\\
&=\{F\in\A^*:(\exists F'\in N_a)(\exists \rho\in G)s^+_\rho F'=F\}\\
&=\{F\in\A^*:(\exists F'\in N_a)(\exists \rho\in G)F'=s^+_{\rho^{-1} }F\}\\
&=\{F\in\A^*:(\exists \rho\in G)s^+_{\rho^{-1} }F\in N_a\}\\
&=\{F\in\A^*:(\exists \rho\in G)a\in s^+_{\rho^{-1} }F\}\\
&=\{F\in\A^*:(\exists \rho\in G)s^+_{\rho }a\in F\}\\
&=\bigcup_{\rho\in G}N_{s^+_\rho a}
\end{align*}
\end{proof}

\begin{theorem} If $\pi$ is closed, then Vaught's conjecture holds.
\end{theorem}
\begin{demo}{Proof} We have ${\cal H}(\A)$ is Borel subset of $\A^*$, the Stone space of $\A$ and ${\cal H}(\A)/E_G$ is a continuous image of
${\cal H}(\A)$.
Because $\pi$ is open, $H(\A)/E$  is second countable.
Now, since $H(\A)$ is metrizable and second countable, it is normal.
But $\pi$ is closed, and so $H(\A)/E$ is also normal, hence regular. Thus ${\cal H}(\A)/E_G$
can also be embedded as an open set in $\mathbb{R}^{\omega},$ hence it is Polish.
If ${\cal H}(\A)/E_G$ is uncountable, then being the continuous image under a map between two Polish spaces of a Borel set, it is analytic.
Then it has the power of the continuum.
\end{demo}

\section{Number of distinguishable models}
In this section we define an equivalence relation on ultrafilters and show that it is Borel.
This implies that it satisfies the Glimm-Effros dichotomy, and so has either countably many or else continuum many equivalence classes.
The equivalence relation we introduce  corresponds to a non-trivial equivalence relation between models which is weaker than isomorphism.
\begin{definition}[Notation]
Let $\mathcal{F}$ be an ultrafilter of a locally finite (cylindric or quasi-polyadic) algebra $\A$.
For $a\in A$ define $$Sat_\mathcal{F}(a)=\{t|_{\Delta a}: t\in {}^{\omega}\omega,\;s^+_t a\in\mathcal{F}\}.$$
\end{definition}

Throughout, $\A$ is countable. We define an equivalence relation $\mathcal{E}$ on the space $\mathcal{H}'(\A)$ (or $\mathcal{H}(\A))$  that turns out to be Borel.

\begin{definition}
Let $\mathcal{E}$ be the following equivalence relation on $\mathcal{H}'(\A)$ (or $\mathcal{H}(\A))$ :\small
$$\mathcal{E}=\{(\mathcal{F}_0,\mathcal{F}_1): (\forall a\in A) (|Sat_{\mathcal{F}_0}(a)|=|Sat_{\mathcal{F}_1}(a)|)\}.$$
\end{definition}

We say that $\mathcal{F}_0,\mathcal{F}_1\in\mathcal{H}'(\A)$(or $\mathcal{H}(\A))$ are \emph{distinguishable } if $(\mathcal{F}_0,\mathcal{F}_1)\notin\mathcal{E}.$ We also say that two models of a theory $T$ are distinguishable if their corresponding ultrafilters in $\mathcal{H}'(CA(T))$(or $\mathcal{H}(QPA(T)))$ are distinguishable. That is, two models are distinguishable if they disagree in the number of realizations they have for some formula.

To show that $\mathcal{E}$ is Borel in the product space $\mathcal{H}'(\A)\times\mathcal{H}'(\A)$(or in $\mathcal{H}(\A)\times\mathcal{H}(\A))$,
we need first to develop some tools that enable us to express \textbf{appropriately} sentences talking about sets. Namely, sentences like ``the following two sets $X,Y$ are of the same size".

Let $K$ be the set of all functions from a finite subset of $\omega$ to $\omega$ and let $\mu$ be a bijection between $\mathbb{N}$ and $K$. Then we can easily see that for a set $X\subset K$:
$$X\mbox{ is infinite \textit{iff}  } (\forall n)(\exists m>n) \mu(m)\in X.$$
Suppose now for $X,Y$ subsets of $K$, we want to say that $|X|=|Y|<\omega$. (Notice that, $|X|=n$, is equivalent to that there is an injective map $f:n\longrightarrow K$ such that $f^*(n)=X$. We show that the following sentence says what we want: \textit{There is } $n<\omega$ \textit{ and there are injective maps} $f,g:n\longrightarrow K$ \textit{such} \textit{that} $f^*(g^{-1}(Y))=X$ and $g^*(f^{-1}(X))=Y$. Indeed, let $Inj(n,K)$ denote the set of all injections from $n$ into $K$. Then,

\begin{align*}
 |X|=|Y|<\omega &\Longleftrightarrow (\exists n)(\exists f,g\in Inj(n,K))(f^*(n)=X\wedge g^*(n)=Y)\\
  &\Longrightarrow (\exists n)(\exists f,g\in Inj(n,K))(f^*(g^{-1}(Y))=X\wedge \\&g^*(f^{-1}(X))=Y) \\
 &\Longrightarrow (\exists n)(\exists f,g\in Inj(n,K))(X\subseteq f^*(n)\wedge\\& Y\subseteq g^*(n)\wedge g^{-1}(Y)=f^{-1}(X))\\
 &\Longrightarrow (\exists n)(\exists f,g\in Inj(n,K))(|X|=|f^{-1}(X)|\wedge\\& |Y|=|g^{-1}(Y)|\wedge |g^{-1}(Y)|=|f^{-1}(X)|)\\
 &\Longrightarrow |X|=|Y|<\omega.
\end{align*}

For $h$, an injective map from a subset of $\omega$ into $\omega$, let $h^{-1}$ denote also the map from $Range(h)$ to $\omega$ that sends $t\in Range(h)$ to the unique element in $h^{-1}(\{t\})$.
Remark that, for $f$ and $g$ like above, because they are injective we have:
\begin{align*}
f^*(g^{-1}(Y))=X &\Longleftrightarrow (\forall t)[t\in X\Leftrightarrow (\exists s)(s\in Y\wedge f(g^{-1}(s))=t)]\\
&\Longleftrightarrow (\forall t)[t\in X\Leftrightarrow (\exists s)(s\in Y\wedge g^{-1}(s)=f^{-1}(t))]\\
&\Longleftrightarrow (\forall t)[t\in X\Leftrightarrow (\exists s)(s\in Y\wedge s=g(f^{-1}(t)))]\\
&\Longleftrightarrow (\forall t)[t\in X\Leftrightarrow g(f^{-1}(t))\in Y].
\end{align*}

Now we carry out a direct usage of the above tools to see that  $\mathcal{E}$ is Borel. In what follows, let $X_a,Y_a$ abbreviate $Sat_{\mathcal{F}_0}(a),Sat_{\mathcal{F}_1}(a)$ respectively and let $N_a$ abbreviate $N_a\cap\mathcal{H}'(\A)$ (or $N_a\cap\mathcal{H}(\A)).$

\begin{align*}
 \mathcal{E} &= \{(\mathcal{F}_0,\mathcal{F}_1): (\forall a\in A) |X_a|=|Y_a|\}\\
  &= \bigcap_{a\in A}\{(\mathcal{F}_0,\mathcal{F}_1): |X_a|=|Y_a|\}\\
  &=\bigcap_{a\in A}[\{(\mathcal{F}_0,\mathcal{F}_1): |X_a|=|Y_a|<\omega\}\cup  \{(\mathcal{F}_0,\mathcal{F}_1): |X_a|,|Y_a|\text{ are both infinite}\}]\\
  &=\bigcap_{a\in A}[\{(\mathcal{F}_0,\mathcal{F}_1):(\exists n)(\exists f,g\in Inj(n,K))(f^*(g^{-1}(Y_a))=X_a\wedge\\& g^*(f^{-1}(X_a))=Y_a)\}\cup\{(\mathcal{F}_0,\mathcal{F}_1):(\forall n)(\exists m,k>n)(\mu(m)\in X_a\wedge \mu(k)\in Y_a )\}]\\
  &=\bigcap_{a\in A}[\bigcup_{n<\omega}\bigcup_{f,g\in Inj(n,K)}\{(\mathcal{F}_0,\mathcal{F}_1):(f^*(g^{-1}(Y_a))=X_a\wedge g^*(f^{-1}(X_a))=Y_a)\}\cup\\ &\bigcap_n\bigcup_{m,k>n}\{(\mathcal{F}_0,\mathcal{F}_1):(\mu(m)\in X_a\wedge \mu(k)\in Y_a )\}]\\
      \end{align*}
  \begin{align*}
  &=\bigcap_{a\in A}[\bigcup_{n<\omega}\bigcup_{f,g\in Inj(n,K)}\{(\mathcal{F}_0,\mathcal{F}_1):(\forall t\in K)(t\in X_a \Leftrightarrow g( f^{-1}(t))\in Y_a)\wedge (\forall t\in K)\\&(t\in Y_a \Leftrightarrow f( g^{-1}(t))\in X_a)\}\cup \bigcap_n\bigcup_{m,k>n}\{(\mathcal{F}_0,\mathcal{F}_1):s_{\mu(m)}a\in\mathcal{F}_0 \wedge s_{\mu(k)}a\in\mathcal{F}_1 \}]\\
  &=\bigcap_{a\in A}[\bigcup_{n<\omega}\bigcup_{f,g\in Inj(n,K)}\bigcap_{t\in K}\{(\mathcal{F}_0,\mathcal{F}_1):(t\in X_a \Leftrightarrow g( f^{-1}(t))\in Y_a)\wedge(t\in Y_a \Leftrightarrow\\& f( g^{-1}(t))\in X_a)\}\cup\bigcap_n\bigcup_{m,k>n}\{(\mathcal{F}_0,\mathcal{F}_1):\mathcal{F}_0 \in N_{s_{\mu(m)}a}\wedge \mathcal{F}_1 \in N_{s_{\mu(k)}a} \}]\\
  &=\bigcap_{a\in A}[\bigcup_{n<\omega}\bigcup_{f,g\in Inj(n,K)}\bigcap_{t\in K}\{(\mathcal{F}_0,\mathcal{F}_1):(s_ta\in\mathcal{F}_0 \Leftrightarrow s_{g( f^{-1}(t))}a\in \mathcal{F}_1)\wedge(s_ta\in\mathcal{F}_1  \Leftrightarrow\\& s_{f( g^{-1}(t))}a\in \mathcal{F}_0)\})\cup\bigcap_n\bigcup_{m,k>n}(N_{s_{\mu(m)}a}\times N_{s_{\mu(k)}a})]\\
  &=\bigcap_{a\in A}[\bigcup_{n<\omega}\bigcup_{f,g\in Inj(n,K)}\bigcap_{t\in K}(\{(\mathcal{F}_0,\mathcal{F}_1):s_ta\in\mathcal{F}_0 \Leftrightarrow s_{g( f^{-1}(t))}a\in \mathcal{F}_1\}\cap\\&\{(\mathcal{F}_0,\mathcal{F}_1):s_ta\in\mathcal{F}_1  \Leftrightarrow s_{f( g^{-1}(t))}a\in \mathcal{F}_0)\})\cup\bigcap_n\bigcup_{m,k>n}(N_{s_{\mu(m)}a}\times N_{s_{\mu(k)}a})]\\
  &=\bigcap_{a\in A}[\bigcup_{n<\omega}\bigcup_{f,g\in Inj(n,K)}\bigcap_{t\in K}(\{(\mathcal{F}_0,\mathcal{F}_1):(s_ta\in\mathcal{F}_0 \wedge s_{g( f^{-1}(t))}a\in \mathcal{F}_1)\vee(s_ta\notin\mathcal{F}_0 \wedge\\& s_{g( f^{-1}(t))}a\notin \mathcal{F}_1)\}\cap\{(\mathcal{F}_0,\mathcal{F}_1):(s_ta\in\mathcal{F}_1  \wedge s_{f( g^{-1}(t))}a\in \mathcal{F}_0)\vee \\&(s_ta\notin\mathcal{F}_1  \wedge s_{f( g^{-1}(t))}a\notin \mathcal{F}_0)\}) \cup\bigcap_n\bigcup_{m,k>n}(N_{s_{\mu(m)}a}\times N_{s_{\mu(k)}a})]\\
  &=\bigcap_{a\in A}[\bigcup_{n<\omega}\bigcup_{f,g\in Inj(n,K)}\bigcap_{t\in K}(((N_{s_ta}\times N_{s_{g(f^{-1}(t))}a})\cup (N_{-s_t(a)}\times N_{-s_{g(f^{-1}(t))}a}))\cap \\&((N_{s_{f(g^{-1}(t))}a}\times N_{s_ta})\cup  (N_{-s_{f(g^{-1}(t))}a}\times N_{-s_ta}))) \cup\bigcap_n\bigcup_{m,k>n}(N_{s_{\mu(m)}a}\times\\& N_{s_{\mu(k)}a})].
\end{align*}

So, $\mathcal{E}$ is Borel. Recall now the following:

If $X$ be a Polish space and $E$ a Borel equivalence relation on $X$. We call $E$ \emph{smooth }if there is a Borel map $f$ from $X$ to the Cantor space ${}^\omega2$ such that $$xEy \Leftrightarrow f(x)=f(y).$$
Note that $E$ is smooth iff $E$ admits a \emph{countable Borel separating family}, i.e., a family $(A_n)$ of Borel sets such that $$xEy\Leftrightarrow \forall n(x\in A_n \leftrightarrow y\in A_n).$$
Clearly, if $E$ is smooth then it is Borel (but the converse is not true).

A standard example of a non-smooth Borel equivalence relation is the following: On $2^\mathbb{N}$, let $E_0$ be defined by $$xE_0y\Leftrightarrow \exists n\forall m\geq n (x(m)=y(m)).$$

We say that the equivalence relation $E$, on a Polish space $X$, satisfies the \emph{Glimm-Effros Dichotomy} if either it is smooth or else it contains a copy of $E_0$. Clearly, for an equivalence relation $E$, $E$ satisfies the Glimm-Effros Dichotomy implies that $E$ satisfies the \emph{Silver-Vaught Dichotomy}, that is, $E$ has either countably many  classes or else perfectly many classes ($X$ has a perfect subset of non-equivalent elements).

\begin{theorem}[Harrington-Kechris-Louveau \cite{HKL}]
Let $X$ be a Polish space and $E$ a Borel equivalence relation on $X$. Then $E$ satisfies the Glimm-Effros Dichotomy.
\end{theorem}
It follows directly from this theorem, replacing $X$ with $\mathcal{H}'(\A)$ (or $\mathcal{H}(\A)$), that $\mathcal{E}$ satisfies the Glimm-Effros dichotomy and so has either countably many equivalence classes or else perfectly many.

\begin{corollary}
Let $T$ be a first order theory  in a countable language (with or without equality). If  $T$ has an uncountable set of countable models that are pairwise distinguishable, then actually it has such a set of size $2^{\aleph_0}$.
\end{corollary}

\begin{remark}
It should be mentioned that, for languages with equality, our last result can be established with less effort.
Here is an argument. Suppose we have a language $L$ with equality.
First note that if  $L^*=L_0\cup L_1$ where $L_0$ and $L_1$ are disjoint copies of $L$,
then $X_{L^*}\cong X_{L_0}\times X_{L_1}$ (where the spaces $X_L$'s are defined as in \cite{BeckerKechris} page 22).

For each formula $\varphi,$ let $\varphi^*$ be the
sentence $\bigwedge_{n\in\omega}(\exists^n \bar{x})\varphi_0(\bar{x})\leftrightarrow (\exists^n \bar{x})\varphi_1(\bar{x})$
where $\varphi_0,$ $\varphi_1$ are the copies of $\varphi$ in $L_0,L_1$ respectively, and $\exists^n$ is a shorthand for
``there exists at least $n$ tuples such that ..."

It is then immediate that two models $M_0,M_1$ of $L$ are not distinguishable iff the model $M$ of $L^*$ such
that $M|_{L_0}=M_0$ and $M|_{L_1}=M_1$ satisfies $\bigwedge_{\varphi\in L}\varphi^*.$
This means that our equivalence relation between models corresponds to the subset of $X_{L^*}$ of models of
the formula $\bigwedge_{\varphi\in L}\varphi^*.$ Such a subset is Borel by Theorem 16.8 in \cite{Kechris}.
\end{remark}

\section{ Number of models omitting a given family of types}

The way we counted the ultrafilters (corresponding to distinguishable models)
above gives a completely analogous  result when we count ultrafilters corresponding  to models {\it omitting} a countable set of non-isolated types.

Given a countable locally finite algebra $\A$, a non-zero $a\in A$ and a non-principal type
$X\subseteq \Nr_n\A$, so that $\prod X=0$, one constructs a model omitting $X$, by finding a Henkin ultrafilter preserving
the following set of infinitary joins and meets where $x\in A$, $i,j\in \omega$ and $\tau$ is a finite transformation:
$c_ix=\sum s_j^ix,$
and $\prod s_{\tau}X=0.$
Working in the Stone space, one finds an ultrafilter in $N_a$ outside the
nowhere dense sets
$N_{i,x}=S\sim \bigcup N_{s_j^i}$ and
$H_{\tau}=\bigcap_{x\in X} N_{\tau}x.$
Now suppose we want to count the number of distinguishable models omitting a
family $\Gamma=\{\Gamma_i:i<\lambda\}$ ($\lambda<covK)$
of non-isolated types of $T$.

Then $$\mathbb{H}=
\mathcal{H}(CA(T))\mbox{(or }\mathcal{H}'(QPA(T)))\cap\sim  \bigcup_{i\in\lambda,\tau\in W}\bigcap_{\varphi\in \Gamma_i}N_{s_\tau (\varphi/\equiv_T)}$$
(where $W=\{\tau\in{}^\omega\omega:|i:\tau(i)\neq i|<\omega\}$) is clearly (by the above discussion)
the space of ultrafilters corresponding to models of $T$ omitting $\Gamma.$

But then by properties of $covK$ union  $\bigcup_{i\in\lambda}$
can be reduced to a countable union.
We then have $\mathbb{H}$ a $G_\delta$ subset of a Polish space. So $\mathbb{H}$ is
Polish and moreover, $\mathcal{E}'=\mathcal{E}\cap (\mathbb{H}\times \mathbb{H})$
is a Borel equivalence relation on $\mathbb{H}.$ It follows then that the number of distinguishable models omitting $\Gamma$
is either countable or else $2^\omega.$
We readily obtain:
\begin{corollary}
Let $T$ be a first order theory  in a countable language (with or without equality).
If  $T$ has an uncountable set of countable models that omit $< covK$ many non principal types that
are pairwise distinguishable, then actually it has such a set of size $2^{\aleph_0}$.
\end{corollary}
Using the same reasoning as above conjuncted with Morleys theorem, we get
\begin{theorem}
The number of countable models of a countable theory that omits $< covK$ many types is either $\leq \omega$ or
$\omega_1$ or $^{\omega}2$.
\end{theorem}


\section{The dimension complemented case}
The following, this time deep theorem, uses ideas of Andr\'eka and N\'emeti, reported in \cite{HMT2}, theorem 3.1.103,
in how to square units of so called weak cylindric set algebras (cylindric algebras whose units are weak spaces):

\begin{theorem}\label{weak} If $\B$ is a subalgebra of $ \wp(^{\alpha}\alpha^{(Id)})$ then there exists a set algebra $\C$ with unit $^{\alpha}U$
such that $\B\cong \C$. Furthermore, the isomorphism is a strong sub base isomorphism.
\end{theorem}

\begin{proof}We square the unit using ultraproducts.
We prove the theorem for $\alpha=\omega$. Let $F$ be a non-principal ultrafilter
over $\omega$. (For $\alpha>\omega$, one takes an $|{\alpha}^+|$ regular ultrafilter on $\alpha$).
Then there exists a function
$h: \omega\to \{\Gamma\subseteq_{\omega} \omega\}$
such that $\{i\in \omega: \kappa\in h(i)\}\in F$ for all $\kappa<\omega$.
Let $M={}^{\omega}U/F$.  $M$ will be the base of our desired algebra, that is  $\C$ will
have unit $^{\omega}M.$
Define $\epsilon: U\to {}^{\omega}U/F$ by
$$\epsilon(u)=\langle u: i\in \omega\rangle/F.$$
Then it is clear that $\epsilon$ is one to one.
For $Y\subseteq {}^{\omega}U$,
let $$\bar{\epsilon}(Y)=\{y\in {}^{\omega}(^{\omega}U/F): \epsilon^{-1}\circ y\in Y\}.$$
By an $(F, (U:i\in \omega), \omega)$ choice function we mean a function
$c$ mapping $\omega\times {}^{\omega}U/F$
into $^{\omega}U$ such that for all $\kappa<\omega$
and all $y\in {}^{\omega}U/F$, we have $c(k,y)\in y.$
Let $c$ be an $(F, (U:i\in \omega), \omega)$
choice function satisfying the following condition:
For all $\kappa, i<\omega$ for all $y\in X$, if
$\kappa\notin h(i)$ then $c(\kappa,y)_i=\kappa$,
if $\kappa\in h(i)$ and $y=\epsilon u$ with  $u\in U$ then $c(\kappa,y)_i=u$.
Let $\delta: \B\to {}^{\omega}\B/F$ be the following monomorphism
$$\delta(b)=\langle b: i\in \omega\rangle/F.$$
Let $t$ be the unique homomorphism
mapping
$^{\omega}\B/F$ into $\wp{}^{\omega}(^{\omega}U/F)$
such that  for any $a\in {}^{\omega}B$
$$t(a/F)=\{q\in {}^{\omega}(^{\omega}U/F): \{i\in \omega: (c^+q)_i\in a_i\}\in F\}.$$
Here $(c^+q)_i=\langle c(\kappa,q_\kappa)_i: k<\omega\rangle.$
It is easy to show that show that $t$ is well-defined. Assume that $J=\{i\in \omega: a_i=b_i\}\in F$. If $\{i\in \omega: (c^+q)_i\in a_i\}\in F$,
then $\{i\in \omega; (c^+q)_i\in b_i\}\in F$. The converse inclusion is the same, and we are done.

Now we check that the map preserves the operations. That the  Boolean operations are preserved is obvious.

So let us check substitutions. It is enough to consider transpositions and replacements.
Let $i,j\in \omega.$  Then $s_{[i,j]}g(a)=g(s_{[i,j]}a)$,
follows from the simple observation that $(c^+q\circ [i,j])_k\in a$ iff $(c^+q)_k\in s_{[i,j]}a$.
The case of replacements is the same;  $(c^+q\circ [i|j])_k\in a$ iff $(c^+q)_k\in s_{[i|j]}a.$

Let $g=t\circ \delta$. Then for $a\in B$, we have
$$g(a)=\{q\in {}^{\omega}(^{\omega}U/F): \{i\in \omega: (c^+q)_i\in a\}\in F\}.$$
Let $\C=g(\B)$. Then $g:\B\to \C$.
We show that $g$ is an isomorphism
onto a set algebra. First it is clear that $g$ is a monomorphism. Indeed if $a\neq 0$, then $g(a)\neq \emptyset$.
Now $g$ maps $\B$ into an algebra with unit $g(V)$.

Recall that $M={}^{\omega}U/F$. Evidently $g(V)\subseteq {}^{\omega}M$.
We show the other inclusion. Let $q\in {}^{\omega}M$. It suffices to show that
$(c^+q)_i\in V$ for all $i\in\omega$. So, let $i\in \omega$. Note that
$(c^+q)_i\in {}^{\omega}U$. If $\kappa\notin h(i)$ then we have
$$(c^+q)_i\kappa=c(\kappa, q\kappa)_i=\kappa.$$
Since $h(i)$ is finite the conclusion follows.
We now prove that for $a\in B$
$$(*) \ \ \ g(a)\cap \bar{\epsilon}V=\{\epsilon\circ s: s\in a\}.$$
Let $\tau\in V$. Then there is a finite $\Gamma\subseteq \omega$ such that
$$\tau\upharpoonright (\omega\sim \Gamma)=
p\upharpoonright (\omega\sim \Gamma).$$
Let $Z=\{i\in \omega: \Gamma\subseteq hi\}$. By the choice of $h$ we have $Z\in F$.
Let $\kappa<\omega$ and $i\in Z$.
We show that $c(\kappa,\epsilon\tau \kappa)_i=\tau \kappa$.
If
$\kappa\in \Gamma,$ then $\kappa\in h(i)$ and so
$c(\kappa,\epsilon \tau \kappa)_i=\tau \kappa$. If $\kappa\notin \Gamma,$
then $\tau \kappa=\kappa$
and $c(\kappa,\epsilon \tau \kappa)_i=\tau\kappa.$
We now prove $(*)$. Let us suppose that $q\in g(a)\cap {\bar{\epsilon}}V$.
Since $q\in \bar{\epsilon}V$ there is an $s\in V$
such that $q=\epsilon\circ s$.
Choose $Z\in F$
such that $$c(\kappa, \epsilon(s\kappa))\supseteq\langle s\kappa: i\in Z\rangle$$
for all $\kappa<\omega$. This is possible by the above.
Let $H=\{i\in \omega: (c^+q)_i\in a\}$.
Then $H\in F$. Since $H\cap Z$ is in $F$
we can choose $i\in H\cap Z$.
Then we have
$$s=\langle s\kappa: \kappa<\omega\rangle=
\langle c(\kappa, \epsilon(s\kappa))_i:\kappa<\omega\rangle=
\langle c(\kappa,q\kappa)_i:\kappa<\omega\rangle=(c^+q)_i\in a.$$
Thus $q\in \epsilon \circ s$. Now suppose that $q=\epsilon\circ s$ with $s\in a$.
Since $a\subseteq V$ we have $q\in \epsilon V$.
Again let $Z\in F$ such that for all $\kappa<\omega$
$$c(\kappa, \epsilon
s \kappa)\supseteq \langle s\kappa: i\in Z\rangle.$$
Then $(c^+q)_i=s\in a$ for all $i\in Z.$ So $q\in g(a).$
Note that $\bar{\epsilon}V\subseteq {}^{\omega}(^{\omega}U/F)$.
Let $rl_{\epsilon(V)}^{\C}$ be the function with domain $\C$
(onto $\bar{\epsilon}(\B))$
such that
$$rl_{\epsilon(V)}^{\C}Y=Y\cap \bar{\epsilon}V.$$
Then we have proved that
$$\bar{\epsilon}=rl_{\bar{\epsilon V}}^{\C}\circ g.$$
It follows that $g$ is a strong sub-base-isomorphism of $\B$ onto $\C$.
\end{proof}
\begin{corollary}
\begin{enumarab}
\item  Let $\alpha$ be any ordinal (possibly uncountable) The logic corresponding to $Dc_{\alpha}$ is complete with respect to ordinary semantics
\item This logic also enjoys an omitting types theorem but with respect to weak models.
\end{enumarab}
\end{corollary}
\begin{proof}
\begin{enumarab}
\item  One forms the Tarski-Lindenbaum algebra $\Fm_T$ of a given consistent theory in a rich language.
Then given non- zero, $a$ one finds an
ultrafilter that contains $\A$, and respects the set of joins, then one defines a homomorphism exactly like the locally finite, except that one
uses only {\it finite} substitution in $\alpha$, so that the target algebra is a weak set algebra in which $f(a)\neq 0$.
In the countable case, this ultrafilter can be found using the Baire category theory, or by a step-by step method.
For the uncountable case, one uses transfinite induction. By the previous theorem,
there is an ordinary set algebra  (that is one with a square unit) that is isomorphic to this last weak set algebra via $g$, say.
Then $g\circ f$ is the desired homomorphism and this finishes the proof.
\item Like the above argument \cite{Samir}.
\end{enumarab}

\end{proof}

Also such languages enjoy an omitting types theorem; for $< covK$ many non-principal types, and the types can contain infinitely many variables
(unlike first order logic) However, the models that omit a countable set of non-principal types is only a weak model,
and it can be proved that there are cases, where it has to be a weak model.
This actually occurs in first order logic, as the following simple example illustrates:
Furthermore, even if we have finite-variable types, then we know that they can be omitted by a weak model,
but the isomorphism constructed above may not to be
a {\it complete} one; so they might not be omitted in a square model.

\begin{example}
Let $T$ be the theory of dense linear order without endpoints.
Then $T$ is complete. Let $\Gamma(x_0, x_1\ldots )$ be the set
$$\{x_1<x_0, x_2<x_1, x_3<x_2\ldots \}.$$ (Here there is no bound on free variables.)
A model $\M$ omits $\Gamma$ if and only if $\M$ is a well ordering.
But $T$ has no well ordered models, so no model of $T$ omits $\Gamma$.
However $T$ locally omits $\Gamma$ because if $\phi(x_0, \ldots x_{n-1})$ is consistent
with $T$, then $\phi\land \neg x_{n+2}<x_{n+1}$ is consistent with $T.$
Note that $\Gamma$ can be omitted in a weak model.
\end{example}
Now let us see  how far we can get, with proving an analogue of counting distinguishable models. We now count
distinguishable {\it weak} models.
Let $\A\in Dc_{\alpha}$. Now we hav only {\it finite} substitutions. As before, let
$$\mathcal{H}(\A)=\bigcap_{i<\omega,x\in A}(N_{-c_ix}\cup\bigcup_{j<\omega}N_{s^i_jx})$$ and, in the cylindric algebraic case, let
and
$$\mathcal{H}'(\A)=\mathcal{H}(\A)\cap\bigcap_{i\neq j\in\omega}N_{-d_{ij}}.$$

Now $\mathcal{H}(\A)$ and $\mathcal{H}'(\A)$ are $G_\delta$ subsets of $\A^*$, and are nonempty, in fact they are dense, and they
are  Polish spaces; (see \cite{Kechris}).
Assume $\mathcal F\in \mathcal{H}(\A).$ For any $x\in A$, define the function
$\mathrm{rep}_{\mathcal F}$ to be
$$\mathrm{rep}_{\mathcal F}(x)=\{\tau\in{}^\omega\omega^{Id}:s_\tau x\in \mathcal F\}.$$

But first a some definitions
\begin{definition}
Let $\A$ and $\B$ be set algebras with bases $U$ and $W$ respectively. Then $\A$ and $\B$
are \emph{base isomorphic} if there exists a bijection
$f:U\to W$ such that $\bar{f}:\A\to \B$ defined by ${\bar f}(X)=\{y\in {}^{\alpha}W: f^{-1}\circ y\in x\}$ is an isomorphism from $\A$ to $\B$
\end{definition}
\begin{definition} An algebra $\A$ is \emph{hereditary atomic}, if each of its subalgebras is atomic.
\end{definition}
Finite Boolean algebras are hereditary atomic of course,
but there are infinite hereditary atomic Boolean algebras; any Boolean algebra generated by by its atoms is
hereditary atomic, for example the finite co-finite
algebra on any set. An algebra that is infinite and complete is not hereditory atomic, wheter atomic or not.

\begin{example}
Hereditary atomic algebras arise naturally as the Tarski Lindenbaum algebras of
certain countable first order theories, that abound. If $T$ is a countable complete first order theory
which has an an $\omega$-saturated model, then for each $n\in \omega$,
the Tarski Lindenbuam Boolean algebra $\Fm_n/T$ is hereditary atomic. Here $\Fm_n$ is the set of formulas using only
$n$ variables. For example $Th(\Q,<)$ is such with $\Q$ the $\omega$ saturated model.
\end{example}

A well known model-theoretic result is that $T$ has an $\omega$ saturated model iff $T$ has countably many $n$ types
for all $n$. Algebraically $n$ types are just ultrafilters in $\Fm_n/T$.
And indeed, what characterizes hereditary atomic algebras is that the base of their Stone space, that is the set of all
ultrafilters, is at most countable.

\begin{lemma}\label{b} Let $\B$ be a countable  Boolean algebra. If $\B$ is hereditary atomic then the number of ultrafilters is at most countable; ofcourse they are finite
if $\B$ is finite. If $\B$ is not hereditary atomic the it has $2^{\omega}$ ultarfilters.
\end{lemma}
\begin{proof}\cite{HMT1} p. 364-365  for a detailed discussion.
\end{proof}
Our next theorem is the, we believe, natural extension of Vaught's theorem to variable rich languages.
However, we address only languages with finitely many relation symbols. (Our algebras are finitely generated,
and being simple, this is equivalent to that it is
generated by a single element.)

\begin{theorem}\label{2} Let $\A\in Dc_{\alpha}$ be countable simple and finitely generated.
Then the number of non-base isomorphic representations of $\A$ is $2^{\omega}$.
\end{theorem}
\begin{proof} Let $V={}^{\alpha}\alpha^{(Id)}$ and let $\A$ be as in the hypothesis. Then $\A$ cannot be atomic \cite{HMT1} corollary 2.3.33,
least hereditary atomic. By \ref{b}, it has $2^{\omega}$ ultrafilters.

For an ultrafilter $F$, let $h_F(a)=\{\tau \in V: s_{\tau}a\in F\}$, $a\in \A$.
Then $h_F\neq 0$, indeed $Id\in h_F(a)$ for any $a\in \A$, hence $h_F$ is an injection, by simplicity of $\A$.
Now $h_F:\A\to \wp(V)$; all the $h_F$'s have the same target algebra.
We claim that $h_F(\A)$ is base isomorphic to $h_G(\A)$ iff there exists a finite bijection $\sigma\in V$ such that
$s_{\sigma}F=G$.
We set out to confirm our claim. Let $\sigma:\alpha\to \alpha$ be a finite bijection such that $s_{\sigma}F=G$.
Define $\Psi:h_F(\A)\to \wp(V)$ by $\Psi(X)=\{\tau\in V:\sigma^{-1}\circ \tau\in X\}$. Then, by definition, $\Psi$ is a base isomorphism.
We show that $\Psi(h_F(a))=h_G(a)$ for all $a\in \A$. Let $a\in A$. Let $X=\{\tau\in V: s_{\tau}a\in F\}$.
Let $Z=\Psi(X).$ Then
\begin{equation*}
\begin{split}
&Z=\{\tau\in V: \sigma^{-1}\circ \tau\in X\}\\
&=\{\tau\in V: s_{\sigma^{-1}\circ \tau}(a)\in F\}\\
&=\{\tau\in V: s_{\tau}a\in s_{\sigma}F\}\\
&=\{\tau\in V: s_{\tau}a\in G\}.\\
&=h_G(a)\\
\end{split}
\end{equation*}
Conversely, assume that $\bar{\sigma}$ establishes a base isomorphism between $h_F(\A)$ and $h_G(\A)$.
Then $\bar{\sigma}\circ h_F=h_G$.  We show that if $a\in F$, then $s_{\sigma}a\in G$.
Let $a\in F$, and let $X=h_{F}(a)$.
Then, we have
\begin{equation*}
\begin{split}
&\bar{\sigma\circ h_{F}}(a)=\sigma(X)\\
&=\{y\in V: \sigma^{-1}\circ y\in h_{F}(X)\}\\
&=\{y\in V: s_{\sigma^{-1}\circ y}a\in F\}\\
&=h_G(a)\\
\end{split}
\end{equation*}
Now we have $h_G(a)=\{y\in V: s_{y}a\in G\}.$
But $a\in F$. Hence $\sigma^{-1}\in h_G(a)$ so $s_{\sigma^{-1}}a\in G$, and hence $a\in s_{\sigma}G$.

Define the equivalence relation $\sim $ on the set of ultrafilters by $F\sim G$, if there exists a finite permutation $\sigma$
such that $F=s_{\sigma}G$. Then any equivalence class is countable, and so we have $^{\omega}2$ many orbits, which correspond to
the non base isomorphic representations of $\A$.
\end{proof}
The above theorem is not so  deep, as it might appear on first reading. The relatively simple
proof is an instance of the obvious fact that if a countable Polish group, acts on an uncountable Polish space, then the number of induced orbits
has the cardinality of the continuum, because it factors out an uncountable set by a countable one. In this case, it is quite easy to show
that  the Glimm-Effros Dichotomy holds.

\begin{theorem}
Let $T$ be a countable theory in a rich language, with only finitely many relation symbols,
and $\Gamma =\{\Gamma_i: i\in covK\}$ be non isolated types.
Then $T$ has $2^{\omega}$ weak models that omit $\Gamma$.
\end{theorem}

Notice that this theorem substantially generalized the theorem in [SayedAMLQ], the latter shows that there exists at least one model omiytting non principal types, this theorem says that there are continuum many of 
them.

\subsection{Omitting types for finite variable fragments}

For finite variable fragments $\L_n$ for $n\geq 3$, the situation turns out to be drastically different.
But first a definition.
\begin{definition}
Assume that $T\subseteq \L_n$. We say that $T$ is $n$ complete iff for all sentences $\phi\in \L_n$ 
we have either $T\models \phi$ or $T\models \neg \phi$. We say that $T$
is $n$ atomic iff for all $\phi\in \L_n$, there is $\psi\in \L_n$ such that $T\models \psi\to \phi$ and for all $\eta\in \L_n$ either $T\models \psi\to \eta$
or $T\models \psi\to \neg \eta.$
\end{definition}

The next theorem \ref{finite} is proved using algebraic logic in \cite{ANT}, 
using combinatorial techniques depending on Ramsey's theorem. 

\begin{theorem}\label{finite} Assume that $\L$ is a countable first order language containing a binary relation symbol. For $n>2$ and 
$k\geq 0$,  there are a consistent $n$ complete
and $n$ atomic theory $T$ using only $n$ variables, and a set $\Gamma(x_1)$ using only $3$ variables
(and only one free variable) such that $\Gamma$ is realized in all models of $T$ but each $T$-witness for $T$ uses
more that $n+k$ variables
\end{theorem}
\begin{demo}{Proof} \cite{sayed}.
\end{demo}

Algebraisations of finite variable fragments of first order logic with $n$ variables is 
obtained from locally finite algebras by truncating the dimensions at $n$.
Expressed, formally this corresponds to the operation of forming $n$ neat reducts.

\begin{definition}\cite{HMT1} Let $\A\in \CA_{\beta}$ and $\alpha<\beta$, then the {\it $\alpha$ neat reduct of $\A$} is the algebra obtained 
from $\A$ by discarding operations in $\beta\sim \alpha$ and restricting the remaining operations 
to the set consisting only of $\alpha$ dimensional elements. An element is $\alpha$ dimensional if its dimension set, $\Delta x=\{i\in \beta: c_ix\neq x\}$ 
is contained in $\alpha$. Such an algebra is denoted by $\Nr_{\alpha}\A$.
\end{definition}
For a class $K\subseteq \CA_{\beta}$, $\Nr_{\alpha}K=\{\Nr_{\alpha}\A: A\in K\}$. It is easy to verify that $\Nr_{\alpha}K\subseteq \CA_{\alpha}$.

A class of particular importance, is the class $S\Nr_{\alpha}\CA_{\alpha+\omega}$ where $\alpha$ is an arbitrary ordinal; 
here $S$ stands for the operation of forming subalgebras. 
This class turns out to be a variety which coincides with
the class of {\it representable} algebras of dimension $\alpha$.

Another class that is of significance is the class $S_c\Nr_{\alpha}CA_{\alpha+\omega}$. Here $S_c$ 
is the operation of forming {\it complete} subalgebras. 
(A Boolean algebra $\A$ is a complete subalgebra of $\B$, if for all $X\subseteq \A$, whenever $\sum X=1$ in $A$, then 
$\sum X=1$ in $\B)$. This class is important because a countable cylindric algebra of dimension $n$ is completey representable if and only if
$\A\in S_c\Nr_n\CA_{\omega}$, for any $\alpha\geq \omega$ and $\A$ is atomic. 
This characterization even works for countable dimensions, 
by modifying the notion of complete representation' relativizing the unit to so-called weak units.

And it turns out for finite  variable fragments, that for a theory $T$ to omit types, 
whether countably or uncountably many, a sufficient condition is that 
the cylindric algebra of formulas $\Fm_T$ is in the class $S_c\Nr_n\CA_{\omega}$. 
Furthermore, the condition of  {\it complete} subalgebras, cannot be omitted.

\begin{lemma} Suppose that $T$ is a theory, 
$|T|=\lambda$, $\lambda$ regular, then there exist models $M_i: i<\chi={}^{\lambda}2$, each of cardinality $\lambda$, 
such that if $i(1)\neq i(2)< \chi$, $\bar{a}_{i(l)}\in |M_{i(l)}|$, $l=1,2,$, $\tp(\bar{a_{l(1)}})=\tp(\bar{a_{l(2)}})$, 
then there are $p_i\subseteq \tp(\bar{a_{l(i)}}),$ $|p_i|<\lambda$ and $p_i\vdash \tp(\bar{a_ {l(i)}}).$
\end{lemma}
\begin{proof} \cite{Shelah} Theorem 5.16
\end{proof}
\begin{corollary} For any countable theory, there is a family of $< {}^{\omega}2$ countable models that overlap only on principla types
\end{corollary}
\begin{theorem}\label{uncountable} Let $\A=S_c\Nr_n\CA_{\omega}$. Assume that $|A|=\lambda$, where $\lambda$ is an ucountable 
cardinal. assume that  $\kappa< {}^{\lambda}2$,
and that $(F_i: i<\kappa)$ is a system of non principal ultrafilters.
Then there exists
a set algebra $\C$ with base $U$ such that $|U|\leq \lambda$, $f:\A\to \C$ such that $f(a)\neq 0$ and for all $i\in \kappa$, $\bigcap_{x\in X_i} f(x)=0.$

\end{theorem}

\begin{proof} Let $\A\subseteq_c \Nr_n\B$, where $\B$ is $\omega$ dimensional, locally finite and has the same cardinality as $\A$. 
This is possible by taking $\B$ to be the subalgebra of which $\A$ is a strong neat reduct generated by  $A$, and noting that we gave countably many 
operations. 
The $F_i$'s correspond to maximal $n$ types in the theory $T$ corresponding to $\B$, that is, the first order theory $T$ such that $\Fm_T\cong \B$. 
Applying Shelah result  let ${\bold F}$ be the given set of non principal ultrafilters, with no model omitting them. 
Then for all $i<{}^{\omega}2$,
there exists $F$ such that $F$ is realized in $\B_i$. Let $\psi:{}^{\omega}2\to \wp(\bold F)$, be defined by
$\psi(i)=\{F: F \text { is realized in  }\B_i\}$.  Then for all $i<{}^{\omega}2$, $\psi(i)\neq \emptyset$.
Furthermore, for $i\neq j$, $\psi(i)\cap \psi(j)=\emptyset,$ for if $F\in \psi(i)\cap \psi(j)$ then it will be realized in
$\B_i$ and $\B_j$, and so it will be principal.of the existence of $^{\lambda}2$ representations of $\B$ and restricting to ultrafilters (maximal types)  in $\Nr_n\B$, 
together  with argument (ii) above, gives a a representation with base $\M$, of the big algebra $\B$, via an injective homomorphism $g$,
omitting the given maximal types. For a sequence $s$ with finite length let $s^+=s\cup id$.
Define $f:\A\to \wp({}^{n}\M)$ via $a\mapsto \{s\in {}^nM: s\cup Id\in f(a)\},$ then clearly $f$ is as desired.
  This implies that $|\bold F|={}^{\omega}2$ which is impossible.
\end{proof}

Now one metalogical reading of the last two theorems is 

\begin{theorem} Let $T$ be an $\L_n$ consistent theory that admits elimination of quantifiers. 
Assume that $|T|=\lambda$ is a regular cardinal.
Let $\kappa<2^{\lambda}$. Let $(\Gamma_i:i\in \kappa)$ be a set of non-principal maximal types in $T$. Then there is a model $\M$ of $T$ that omits
all the $\Gamma_i$'s
\end{theorem}
\begin{demo}{Proof} If $\A=\Fm_T$ denotes the cylindric algebra corresponding to $T$, then since $T$ admits elimination of quantifiers, then
$\A\in \Nr_n\CA_{\omega}$. This follows from the following reasoning. Let $\B=\Fm_{T_{\omega}}$ be the locally finite cylindric algebra
based on $T$ but now allowing $\omega$ many variables. Consider the map $\phi/T\mapsto \phi/T_{\omega}$. 
Then this map is from $\A$ into $\Nr_n\B$. But since $T$ admits elimination of quantifiers the map is onto.
The Theorem now follows.
\end{demo}

We now give another natural omitting types theorem for certain uncountable languages.
Let $L$ be an ordinary first order language with 
a list $\langle c_k\rangle$ of individual constants
of order type $\alpha$. $L$ has no operation symbols, but as usual, the
list of variables is of order type $\omega$. Denote by $Sn^{L_{\alpha}}$ the set
of all $L$ sentences, the subscrpt $\alpha$ indicating that we have $\alpha$ many constants Let $\alpha=n\in \omega$. 
Let $T\subseteq Sn^{L_0}$ be consistent. Let $\M$  be an $\L_0$  model of $T$.  Then any $s:n\to M$ defines an expansion of $\M$ to $L_n$ which we denote by $\M{[s]}$.  
For $\phi\in L_n$ let $\phi^{\M}=\{s\in M^n: \M[s]\models \phi\}$. Let 
$\Gamma\subseteq Sn^{L_{n}}$. 
The question we adress is: Is there a model $\M$ of $T$ such that for no expansion $s:n\to M$ we have 
$s\in  \bigcap_{\phi\in \Gamma}\phi^{M}$. 
Such an $\M$ omits $\Gamma$. Call $\Gamma$ principal over $T$ if there exists $\psi\in L_n$ consistent with $T$ such that $T\models \psi\to \Gamma.$ 
Other wise $\Gamma$ is non principal over T.

\begin{theorem}  Let $T\subseteq Sn^{L_0}$ be consistent and assume that $\lambda$ is a regular cardinal, and $|T|=\lambda$. 
Let $\kappa<2^{\lambda}$. Let $(\Gamma_i:i\in \kappa)$ be a set of non-principal maximal types in $T$. 
Then there is a model $\M$ of $T$ that omits
all the $\Gamma_i$'s
That is, there exists a model $\M\models T$ such that is no $s:n\to \M$
such that   $s\in \bigcap_{\phi\in \Gamma_i}\phi^{\M}$. 
\end{theorem}
\begin{demo}{Proof}
Let $T\subseteq Sn^{L_0}$ be consistent. Let $\M$ be an $\L_0$  model of $T$. 
For $\phi\in Sn^L$ and $k<\alpha$
let $\exists_k\phi:=\exists x\phi(c_k|x)$ where $x$ is the first variable
not occuring in $\phi$. Here $\phi(c_k|x)$ is the formula obtained from $\phi$ by 
replacing all occurences of $c_k$, if any, in $\phi$ by $x$.
Let $T$ be as indicated above, i.e $T$ is a set of sentences in which no constants occur. Define the
equivalence relation $\equiv_{T}$ on $^2Sn^L$ as follows
$$\phi\equiv_{T}\psi \text { iff } T\models \phi\equiv \psi.$$
Then, as easily checked $\equiv_{T}$ is a 
congruence relation on the algebra 
$$\Sn=\langle Sn,\land,\lor,\neg,T,F,\exists_k, c_k=c_l\rangle_{k,l<n}$$  
We let $\Sn^L/T$ denote the quotient algebra.
In this case, it is easy to see that $\Sn^L/T$ is a $\CA_n$, in fact is an $\RCA_n$.
Let $L$ be as described in  above. But now we denote it $L_n$, 
the subsript $n$ indicating that we have $n$-many 
individual constants. Now enrich $L_{n}$ 
with countably many constants (and nothing else) obtaining
$L_{\omega}$. 
Recall that both languages, now, have a list of $\omega$ variables. 
For $\kappa\in \{n, \omega\}$ 
let $\A_{\kappa}=\Sn^{L_{k}}/{T}$. 
For $\phi\in Sn^{L_n}$, let $f(\phi/T)=\phi/{T}$. 
Then, as easily checked $f$ is an embedding  of $\A_{n}$ 
into $\A_{\omega}$. Moreover $f$ has the additional property that
it maps $\A_{n}$, into (and onto) the neat $n$ reduct of $\A_{\beta}$,
(i.e. the set of $\alpha$ dimensional elements of $A_{\beta}$). 
In short, $\A_{n}\cong \Nr_{n}\A_{\omega}$. Now agin putting $X_i=\{\phi/T: \phi\in \Gamma_i\}$ and using that the 
$\Gamma_i$'s are maximal non isolated, it follows that the 
$X_i's$ are non-principal ultrafilters
Since $\Nr_n\CA_{\omega}\subseteq S_c\Nr_n\CA_{\omega}$, then our result follows, also from Theorem 1.    
\end{demo}

\end{document}